\documentclass[reqno, 12pt]{amsart}
\pdfoutput=1
\makeatletter
\let\origsection=\section \def\section{\@ifstar{\origsection*}{\mysection}} 
\def\mysection{\@startsection{section}{1}\z@{.7\linespacing\@plus\linespacing}{.5\linespacing}{\normalfont\scshape\centering\S}}
\makeatother        
\usepackage{amsmath,amssymb,amsthm}    
\usepackage{mathrsfs}    
\usepackage{mathabx}\changenotsign

\usepackage{microtype}

\usepackage{xcolor}
\usepackage[backref]{hyperref}
\hypersetup{
    colorlinks,
    linkcolor={red!60!black},
    citecolor={green!60!black},
    urlcolor={blue!60!black}
}

\usepackage{bookmark}

\usepackage[abbrev,msc-links,backrefs]{amsrefs}

\usepackage{doi}

\renewcommand{\PrintDOI}[1]{\doi{#1}}

\usepackage[T1]{fontenc}
\usepackage{lmodern}

\usepackage[english]{babel}

\usepackage{geometry}
\usepackage{enumitem}
\def\rmlabel{\upshape({\itshape \roman*\,})}
\def\RMlabel{\upshape(\Roman*)}
\def\alabel{\upshape({\itshape \alph*\,})}

\let\polishlcross=\l
\def\l{\ifmmode\ell\else\polishlcross\fi}

\def\qand{\quad\text{and}\quad}
\def\qqand{\qquad\text{and}\qquad}

\let\emptyset=\varnothing
\let\setminus=\smallsetminus

\makeatletter
\def\moverlay{\mathpalette\mov@rlay}
\def\mov@rlay#1#2{\leavevmode\vtop{   \baselineskip\z@skip \lineskiplimit-\maxdimen
   \ialign{\hfil$\m@th#1##$\hfil\cr#2\crcr}}}
\newcommand{\charfusion}[3][\mathord]{
    #1{\ifx#1\mathop\vphantom{#2}\fi
        \mathpalette\mov@rlay{#2\cr#3}
      }
    \ifx#1\mathop\expandafter\displaylimits\fi}
\makeatother

\newcommand{\dcup}{\charfusion[\mathbin]{\cup}{\cdot}}

\newtheorem{theorem}{Theorem} 
\newtheorem{lemma}[theorem]{Lemma}

\newtheorem{claim}[theorem]{Claim} 
\newtheorem{fact}[theorem]{Fact}

\newtheoremstyle{definition}  {4pt}  {4pt}  {\sl}  {}  {\bfseries}  {.}  {.5em}          {}
\theoremstyle{definition}
\newtheorem{definition}[theorem]{Definition}

\theoremstyle{remark}

\newtheoremstyle{introthms}  {3pt}  {3pt}  {\itshape}  {}  {\bfseries}  {.}  {.5em}          {\thmnote{#3}}\theoremstyle{introthms}

\let\eps=\varepsilon
\let\theta=\vartheta
\let\rho=\varrho
\let\phi=\varphi

\def\cC{\mathcal{C}}

\def\cO{\mathcal{S}}
\def\cS{\mathcal{O}} 
\def\cI{\mathcal{I}}

\DeclareMathOperator{\vol}{vol}

\newcommand{\vek}[1]{\boldsymbol{#1}}

\newcommand{\pr}[1]{\mathbb{P}\left[#1\right]}
\newcommand{\ex}[1]{\mathbb{E}#1}

\DeclareMathOperator{\ext}{ex}

\newcommand{\slanted}[1]{\slshape{#1}}

\newcommand{\raf}[1]{\eqref{#1}}

\newcommand{\iso}{\cong}

\begin{document}

\title{Extremal results for odd cycles in sparse pseudorandom graphs}

\author{Elad Aigner-Horev}
\address{Fachbereich Mathematik, Universit\"at Hamburg, Hamburg, Germany}
\email{elad.horev@math.uni-hamburg.de}

\author{Hi\d{\^e}p H\`an}
\address{Instituto de Matem\'atica e Estat\'{i}stica, Universidade de S\~{a}o Paulo,S\~{a}o Paulo, Brazil}
\email{hh@ime.usp.br}

\author[Mathias Schacht]{Mathias Schacht}
\address{Fachbereich Mathematik, Universit\"at Hamburg, Hamburg, Germany}
\email{schacht@math.uni-hamburg.de}
\thanks{E.\ Aigner-Horev was supported by the Minerva Foundation and the Young Researchers Initiative of the University of Hamburg.
H.\ H\`an was supported by FAPESP (Proc.~2010/16526-3), by CNPq (Proc.~477203/2012-4) and by NUMEC/USP, N\'ucleo de Modelagem Estoc\' astica e Complexidade of the University of S\~ao Paulo. M.\ Schacht was supported through the Heisenberg-Programme of the
DFG\@.}

\begin{abstract}
We consider extremal problems for subgraphs of pseudorandom graphs. For graphs $F$ and $\Gamma$
the generalized Tur\'an density $\pi_F(\Gamma)$ denotes the density of a maximum 
subgraph of $\Gamma$, which contains no copy of~$F$. 
Extending classical Tur\'an type results for odd cycles, we show that
$\pi_{F}(\Gamma)=1/2$ provided~$F$ is an odd cycle and $\Gamma$ is a sufficiently pseudorandom graph.

In particular, for $(n,d,\lambda)$-graphs $\Gamma$, i.e., $n$-vertex, $d$-regular graphs with
all non-trivial eigenvalues in the interval~$[-\lambda,\lambda]$, 
our result holds for odd cycles of length~$\l$, provided
\[
\lambda^{\l-2}\ll \frac{d^{\l-1}}n\log(n)^{-(\l-2)(\l-3)}\,.
\]
Up to the polylog-factor this  
verifies a conjecture of Krivelevich, Lee, and Sudakov.
For triangles the condition is best possible 
and was proven previously by 
Sudakov, Szab\'o, and Vu, 
who
addressed the case when $F$ is a complete graph. 
A construction of Alon and Kahale (based on an earlier construction of Alon for triangle-free $(n,d,\lambda)$-graphs) 
shows that our assumption on 
$\Gamma$ is best possible up to the polylog-factor for every odd $\l\geq 5$.
\end{abstract}
\maketitle

\section{Introduction and main result}

For two graphs $G$ and $H$, the \emph{generalized Tur\'an number}, denoted by $\ext(G,H)$, is defined to be the maximum number of edges  an $H$-free  subgraph of~$G$ may have. Here, a graph~$G$ is $H$-free if it contains no
 copy of $H$ as a (not necessarily induced) subgraph. With this notation, the well known result of 
 Erd\H os and Simonovits~\cite{ErdSim} (which already appeared implicitly in the work of  
 Erd\H{o}s and Stone~\cite{ErdSt}) reads
\begin{equation}\label{eq:ES-thm}
\ext(K_n, H)=\left(1-\frac1{\chi(H)-1}+o(1)\right)\binom{n}2
\end{equation}
where $\chi(H)$ denotes the chromatic number of $H$.

The systematic study of extensions of~\eqref{eq:ES-thm} arising from replacing~$K_n$ 
with a sparse random or a pseudorandom graph was initiated by
Kohayakawa and collaborators (see, e.g.,~\cites{HKL95,HKL96,KLR97,KRSSS}, see also~\cite{BSS90} for some earlier work for $G(n,1/2)$).
For random graphs such extensions were obtained recently in~\cites{CG,Sch}. 

Here, we continue the study 
for pseudorandom graphs. Roughly speaking, a pseudorandom graph is 
a graph whose edge distribution closely resembles 
that of a truly random graph of the same edge density. 
One way to formally capture such a notion of pseudorandomness is through {\slanted eigenvalue separation}. 
A graph $G$ on $n$ vertices may be associated with an $n \times n$ adjacency matrix $A$. 
This matrix is symmetric and, hence, all its eigenvalues $\lambda_1\geq\lambda_2\geq\dots\geq \lambda_n$ are real.
If $G$ is $d$-regular, then $\lambda_1=d$ and $|\lambda_n|\leq d$ and
the difference in order of magnitude between $d$ and the \emph{second eigenvalue}  
$\lambda(G)=\max\{\lambda_2,|\lambda_n|\}$ of~$G$ is often called the \emph{spectral gap} of~$G$. It is well known~\cites{A86,T} that the spectral gap provides a measure of control over the edge distribution of $G$. Roughly, the larger is the spectral gap the stronger is the resemblance  between the edge distribution of $G$ and that of the random graph $G(n,p)$, where $p=d/n$.  
This phenomenon led to the notion of $(n,d,\lambda)$-\emph{graphs} by which we mean  $d$-regular 
$n$-vertex graphs satisfying  $|\lambda(G)| \leq \lambda$. 

Tur\'an type problems for sparse pseudorandom graphs were addressed in~\cites{Chung,KRSSS,SSzV}. 
In this paper, we continue in studying extensions of the Erd\H{o}s-Stone~\cite{Bol} theorem for sparse host 
graphs and determine upper bounds for the generalized Tur\'an number for odd cycles in sparse pseudorandom 
host graphs, i.e., $\ext(G,C_{2k+1})$ where $G$ is a pseudorandom graph and $C_{2k+1}$ is the odd cycle of 
length $2k+1$.
Our work is related to  work of Sudakov, Szab\'o, and Vu~\cite{SSzV} who investigated $\ext(G,K_t)$  
for a pseudorandom graph $G$ and the clique $K_t$ of order $t\geq 3$. 
Their result may be viewed as the pseudorandom counterpart 
of Tur\'an's theorem~\cite{Turan}. 
For any graph $G$, the trivial lower bound $\ext(G,C_{2k+1})\geq |E(G)|/2$ 
follows from the fact that every graph $G$ contains a bipartite subgraph with at least half the edges of $G$. 
For~$G \iso K_n$, this bound is essentially tight, by the Erd\H{o}s-Stone theorem. Our  result asserts that 
this bound remains essentially tight for sufficiently pseudorandom graphs. 

\begin{theorem}
\label{thm:ndlambda}
Let  $k\geq 1$ be an integer. If $\Gamma$ is an $(n,d,\lambda)$-graph satisfying
\begin{align} 
\lambda^{2k-1} & \ll \frac{d^{2k}}{n} \left( \log n \right)^{-2(k-1)(2k-1)} \label{lambda-bound}, \\
\intertext{then}
\ext(\Gamma,C_{2k+1}) & =\left (\frac12+o(1)\right)\frac{dn}2.  \end{align}
\end{theorem}
The asymptotic notation in \raf{lambda-bound} means that $\lambda$ and $d$ may be functions of $n$ and for two functions $f(n)$, $g(n) >0$ we write $f(n) \ll g(n)$ whenever $\frac{f(n)}{g(n)} \rightarrow 0$ as $n \rightarrow \infty$.

For $k=1$, the same problem was studied in~\cite{SSzV}. In this case, we obtain the same result  which is known to be best possible due to
a construction of Alon~\cite{A94}. 
For $k\geq 2$, Alon's construction can be extended as to fit for general odd cycles; see~\cite{AK} and~\cite{KS}*{Example~10}, implying
 that  for any $k\geq2$ the condition~\raf{lambda-bound} is best possible up to the polylog-factor. We also remark that Theorem~\ref{thm:ndlambda} was essentially (up to the polylog-factor) conjectured by Krivelevich, 
Lee, and Sudakov~\cite{KLS10}*{Conjecture~7.1}.

\subsection{The main result}

Theorem~\ref{thm:ndlambda} is a consequence of Theorem~\ref{thm:main} stated below for the so called \emph{jumbled} graphs. We recall this notion of pseudorandomness 
which can be traced back to the work of Thomason~\cite{Th}.

Let $\Gamma$ be a graph and $X, Y\subset V(\Gamma)$. By $\vol(X,Y)$ we denote the number of all pairs 
with one element from $X$ and the other element being from $Y$ and by $e_{\Gamma}(X,Y)$ we denote the number of actual edges $xy \in E(\Gamma)$ satisfying $x \in X$ and $y \in Y$
and  let $e_{\Gamma}(X)=e_{\Gamma}(X,X)$.
\begin{definition}
Let $p=p(n)$ be a sequence of densities, i.e., $0\leq p\leq 1 $, and let $\beta=\beta(n)$.
An $n$-vertex graph $\Gamma$ is called $(p,\beta)$-\emph{jumbled} if 
\begin{equation*}
\left|e_{\Gamma}(X,Y)-p\vol(X,Y)\right| \leq \beta \vol(X,Y)^{1/2},
\end{equation*}
for all $X,Y \subseteq V(\Gamma).$
\end{definition}
In particular, for disjoint sets $X,Y$ 
\begin{equation}\label{eq:jumbled1}
\left|e_{\Gamma}(X,Y)-p|X||Y|\right| \leq \beta (|X||Y|)^{1/2},
\end{equation}
and for $X=Y$ we have
\begin{equation}\label{eq:jumbled2}
\left|e_{\Gamma}(X)-p\binom{|X|}2\right| \leq \beta |X|. 
\end{equation}

The following is our main result. 
\begin{theorem}
\label{thm:main}
For every integer $k\geq 1$ and every $\delta>0$ there exists a $\gamma>0$ such that for every sequence of densities
$p=p(n)$ there exists an $n_0$ such that for any $n \geq n_0$ the following holds.

If $\Gamma$ is an $n$-vertex $(p,\beta)$-jumbled graph satisfying
\begin{align}
\beta=\beta(n) & \leq \gamma p^{1+\frac1{(2k-1)}}n\log^{-2(k-1)} n, \label{eq:beta} \\
\intertext{then} 
\ext(\Gamma,C_{2k+1}) & < \left(\frac12+\delta\right)p\binom{n}2. \nonumber
\end{align}
\end{theorem}

By the so called \emph{expander mixing lemma}~\cites{AM,T} (see also~\cite{AS}*{Prop.~9.2.1}), 
an $(n,d,\lambda)$-graph  is  
$(p,\beta)$-jumbled with $p=d/n$ and 
$\beta=\lambda$ and it is easily seen that Theorem~\ref{thm:main} indeed implies Theorem~\ref{thm:ndlambda}.

\section{Proof of Theorem~\ref{thm:main}}\label{sec:proof-main}

Our proof of Theorem~\ref{thm:main} relies on Lemmas~\ref{lem:partition} and~\ref{lem:main} stated below. In this section, we state these lemmas while defering their proofs to Sections~\ref{sec:partitionlemma} and~\ref{sec:mainlemma}, respectively. We then show how these two lemmas imply Theorem~\ref{thm:main}. 

To state Lemma~\ref{lem:partition}, we employ the following notation. 
For a graph $G$ and disjoint vertex sets $X,Y \subseteq V(G)$, we write $G[X,Y]$ to denote the bipartite subgraph of $G$ whose vertex set is $X \cup Y$ and whose edge set, denoted $E_G(X,Y)$, consists of all edges of $G$ with one end in $X$ and the other in $Y$.  Also, we write $E_G(X)$ to denote the edge set of $G[X]$. 

For a graph $R$ and a positive integer $m$, we write $R(m)$ to denote the graph obtained by replacing every vertex $i\in V(R)$ with a set of vertices $V_i$ of size $m$ and adding the complete bipartite graph between $V_i$ and $V_j$ whenever $ij\in E(R)$. A spanning subgraph of $R(m)$ is called  an $R(m)$-\emph{graph}. In addition, such a graph, say $G \subseteq R(m)$, is called $(\alpha,p,\eps)$-\emph{degree-regular} 
if $\deg_{G[V_i,V_j]}(v)=(\alpha \pm \eps)pm$ holds whenever $ij \in E(R)$ and $v \in V_i \cup V_j$. 
Here and in what follows the term $(a\pm b)$ stands for a number in the interval $[a-b,a+b]$.
Throughout, the notation $R(m')$ for a positive real $m'$ is shorthand for $R(\lceil m' \rceil)$; such conventions do not effect our asymptotic estimates. 

The following lemma essentially asserts that under a certain assumption of jumbledness, a relatively dense subgraph of a sufficiently large $(p,\beta)$-jumbled graph contains a  degree-regular $C_{\l}(m)$-graph with large $m$.

\begin{lemma}
\label{lem:partition}
For any integer $\l\geq 3$, all $\rho> 0$, $\alpha_0>0$ and $0<\eps<\alpha_0$ there exist a $\nu>0$ and  a $\gamma>0$ such that for every sequence of densities $p=p(n)\gg\log n/n$ there exists an~$n_0$ such that for every $n \geq n_0$ the following holds.

Let $\Gamma$ be an $n$-vertex $(p,\beta)$-jumbled graph with $\beta=\beta(n) \leq \gamma p^{1+\rho}n$ and let 
$G\subset \Gamma$ be a subgraph of $\Gamma$ satisfying $e(G)\geq \alpha_0 p\binom{n}2$.
Then, there exists an $\alpha \geq \alpha_0$ such that $G$ contains an $(\alpha,p,\eps)$-degree-regular 
$C_{\ell}(\nu n)$-graph as a subgraph.
\end{lemma}

Equipped with Lemma~\ref{lem:partition}, we focus on large degree-regular $C_{\l}(m)$-graphs hosted in a 
sufficiently jumbled graph $\Gamma$. In this setting, we shall concentrate on odd cycles in $\Gamma$ 
that have all but one of their edges in the hosted $C_{\l}(m)$-graph. The remaining edge belongs to~$\Gamma$.
The first part of Lemma~\ref{lem:main} stated below provides a lower bound for the number of such 
configurations (see \raf{eq:wishfor1}). Before making this precise we require some additional notation.
Fix a vertex labeling of $C_{2k+1}$, say, $(u_k,\dots,u_1,w,v_1,\dots,v_k)$ and for a given graph $\Gamma$
let $H \subseteq \Gamma$ be a $C_{2k+1}(m)$-graph with 
the  corresponding vertex partition $(U_k,\dots,U_{1},W,V_1,\dots,V_k)$. 
By $\cC(H,\Gamma)$ we denote the set of all cycles of length $(2k+1)$ of the form 
$(u'_k,\dots,u'_1,w',v'_{1},\dots,v'_k)$ such that $w'\in W, v'_i\in V_i$,  
$u'_i\in U_i$, $v'_k u'_k \in E(\Gamma)$, and all edges other than $v'_k u'_k$ are in $E(H)$. 
In other words, a member of $\cC(H,\Gamma)$ is a cycle of~$\Gamma$ of length $2k+1$ which respects the 
vertex partition of $H$ and which has all edges but possibly $v'_k u'_k$ being in $H$.

For a real number $\mu>0$, an edge of $\Gamma[V_k,U_k] $  is called 
$\mu$-\emph{saturated} if such is contained in at least $p(\mu pm)^{2k-1}$ members of $\cC(H,\Gamma)$.
A cycle in $\cC(H,\Gamma)$ containing a $\mu$-saturated edge is called a \emph{$\mu$-saturated cycle}.
We write $\cO(\mu,H,\Gamma)$ to denote the set of $\mu$-saturated cycles in $\cC(H,\Gamma)$.
To motivate the definition of $\mu$-saturated edges, note that we expect that an edge of~$\Gamma[U_k,V_k]$ extends to $(\alpha p)^{2k}m^{2k-1}$ members of $\cC(H,\Gamma)$. For $\mu \approx \alpha$, a $\mu$-saturated edge overshoots this expectation by a factor of $1/\alpha$. 
 
The following lemma asserts that for sufficiently jumbled graphs the number of $\alpha$-saturated cycles is negligible compared to $|\cC(H,\Gamma)|$ (see \raf{eq:wishfor1} and \raf{eq:wishfor2}).

\begin{lemma}
\label{lem:main}
For any integer $k\geq1$ and all reals  $0<\nu,\alpha_0 \leq 1$, and $0< \eps \leq \alpha_0/3 $ there exists a $\gamma>0$  such that for every sequence of densities $p=p(n)$ there exists an $n_0$ such that for any $n \geq n_0$ the following holds. 

If  $\Gamma$ is an $n$-vertex $(p,\beta)$-jumbled  graph with 
\begin{equation}
\beta=\beta(n)\leq \gamma p^{1+\frac1{2k-1}}n\log^{-2(k-1)} n,
\end{equation}
then for any $m\geq \nu n$ and any $\alpha \geq \alpha_0$ an $\left(\alpha,p,\eps \right)$-degree-regular $C_{2k+1}(m)$-graph $H\subseteq \Gamma$ satisfies 
\begin{align}
\label{eq:wishfor1}
|\cC(H,\Gamma)| & \geq  \left(\alpha-2\eps\right)^{2k}(pm)^{2k+1}\\
\intertext{and}
\label{eq:wishfor2}
|\cO(\alpha+2\eps,H,\Gamma)| & \leq  (3\eps)^{2k}\left(pm\right)^{2k+1}.
\end{align}
\end{lemma}

With Lemmas~\ref{lem:partition} and~\ref{lem:main} stated above (and proved in Sections~\ref{sec:partitionlemma} 
and~\ref{sec:mainlemma}, respectively) we proceed by showing how they imply 
Theorem~\ref{thm:main}. 

\begin{proof}[Proof of Theorem~\ref{thm:main}]
Let $k\geq 1$ and $\delta>0$ be given. Without loss of generality, we may assume that $\delta\leq 1/2$. We set
\begin{equation}\label{eq:constants}
\l=2k+1\,,\quad\rho = \l^{-1}\,,\quad\eps =\frac{\delta}{4+32k+6^{2k+1}}\,,\qand
\alpha_0 =1/2+\delta\,,
\end{equation}
and let $\nu$ and $\gamma_1$ be those obtained by applying Lemma~\ref{lem:partition} with $\l,\rho,\eps$, and $\alpha_0$
as these are set in \raf{eq:constants}. Next, let $\gamma_2$ be that obtained by applying Lemma~\ref{lem:main} with $k,\nu,\alpha_0$ and $\eps$, and set 
\begin{equation}\label{eq:gamma}
\gamma = \min \{\gamma_1,\gamma_2,\delta \nu / 4\}.
\end{equation} 
From this point on, Theorem~\ref{thm:main} and Lemmas~\ref{lem:partition} and~\ref{lem:main} are quantified in a similar manner. Hence, given $p =p(n)$, let $n_0$ be sufficiently large as to accommodate Lemmas~\ref{lem:partition} and~\ref{lem:main}.

Let $\Gamma$ be a $(p,\beta)$-jumbled $n$-vertex graph with $\beta$ satisfying \raf{eq:beta}.
To prove Theorem~\ref{thm:main}, it is sufficient to show that every subgraph $G$ of $\Gamma$ satisfying 
$e(G) \geq \alpha_0 p \binom{n}2$ contains a $C_{2k+1}$.  To that end, let $G$ be such a subgraph of $\Gamma$ and let 
$H \subseteq G$ be an $(\alpha,p,\eps)$-degree-regular $C_{2k+1}(m)$-graph given by Lemma~\ref{lem:partition}, where 
$\alpha \geq \alpha_0$ and $m\geq  \nu n$. Next, let $F=F(U_k,V_k) \subseteq E_{\Gamma}(U_k,V_k)$ denote those edges of 
$\Gamma[U_k,V_k]$ met by a member of $\cC(H,\Gamma)$. 
Every edge in $F$ completes a path of length $2k$ in $H$ into a cycle of length $2k+1$. 
In what follows, we prove that $F \cap E_H(U_k,V_k) \not= \emptyset$ which then implies that 
$C_{2k+1} \subseteq H \subseteq G$ completing the proof of Theorem~\ref{thm:main}.

To this end, it is sufficient to show
\begin{equation}\label{eq:F}
|F| \geq \left( \alpha -\frac{\delta}{2} \right) pm^2.
\end{equation}
Indeed, as $H$ is $(\alpha,p,\eps)$-degree-regular and since $\eps< \delta/2$, we obtain
\[
e_H(U_k,V_k) = 
\sum_{v \in U_k}\deg_{H[U_k,V_k]}(v) \geq 
\left(\alpha-\frac{\delta}{2}\right)pm^2.
\]
On the other hand, jumbledness of $\Gamma$ combined with $m\geq \nu n$ and $\gamma \leq \delta \nu/4$ guarantees 
\[e_{\Gamma}(U_k,V_k) \leq pm^2 +\beta m \leq pm^2+\gamma p^{1+\rho} n m \leq \left(1+\delta/2\right)pm^2.
\]
Hence, $e_{\Gamma}(U_k,V_k) <|F|+e_H(U_k,V_k)$
implying that   $|F|$ and  $E_H(U_k,V_k)$ have non-empty intersection.

It remains to show \raf{eq:F}. By definition, each member of $F$ is contained in at most 
$p \left(\alpha+2 \eps \right)^{2k-1} (pm)^{2k-1}$ members of 
$\cC' = \cC(H,\Gamma) \setminus \cO(\alpha+2\eps,H,\Gamma)$ so that  
\begin{equation}\label{eq:F2}
|F| \geq \frac{|\cC'|}{p \left(\alpha+2 \eps \right)^{2k-1} (pm)^{2k-1}}.
\end{equation}
Next, owing to Lemma~\ref{lem:main}, we obtain the estimates
\begin{align}
|\cC(H,\Gamma)| &  \geq \left(\alpha - 2\eps \right)^{2k} (pm)^{2k+1} \label{eq:cC} \\
\intertext{and}
|\cO(\alpha+2\eps,H,\Gamma)| & \leq (3\eps)^{2k}(pm)^{2k+1} \label{eq:cO}.
\end{align}
Combining \raf{eq:F2}, \raf{eq:cC}, and \raf{eq:cO} we obtain that 
$$
|F| \geq \frac{(\alpha-2\eps)^{2k} - (3\eps)^{2k}}{(\alpha+2\eps)^{2k-1}} pm^2.
$$

First, we consider the term 
$$
T = \frac{(\alpha-2\eps)^{2k}}{(\alpha+2\eps)^{2k-1}} = (\alpha-2\eps) 
\left(\frac{\alpha-2\eps}{\alpha+2\eps} \right)^{2k-1}. 
$$
As $\frac{a-b}{a+b} \geq 1 - \frac{2}{a}b$ for any  $a,b>0$ and $\alpha\geq 1/2$, we attain
$$
T \geq (\alpha-2\eps) \left(1- \frac{2}{\alpha}2\eps \right)^{2k-1} \geq (\alpha-2\eps) (1-8\eps)^{2k-1}.
$$
By the Bernoulli inequality $(1-8 \eps)^{2k-1} \geq (1-16 k\eps)$. Then,  since $\delta \leq 1/2$ it follows that
\begin{equation}\label{eq:T}
T \geq \alpha-2\eps-16k\eps.
\end{equation}

Combining the fact that 
$$
\frac{(3\eps)^{2k}}{(\alpha+2\eps)^{2k-1}} \leq 2^{2k-1}3^{2k}\eps \leq 6^{2k} \eps
$$
with \raf{eq:T}, we arrive at 
$$
|F| \geq \left( \alpha-2\eps-16k\eps-6^{2k}\eps \right)pm^2.
$$
As $2\eps +16k \eps+ 6^{2k}\eps \leq \delta/2$, by the choice of $\eps$, \raf{eq:F} follows. 
\end{proof}

\section{Proof of Lemma~\ref{lem:partition}}
\label{sec:partitionlemma}

In this section, we prove Lemma~\ref{lem:partition}. This lemma follows from Lemma~\ref{lem:cleanup} below. Roughly speaking, Lemma~\ref{lem:cleanup} 
asserts that under certain assumptions, a jumbled graph contains a large subgraph with all its 
vertices having almost the same degree.

\begin{lemma}
\label{lem:cleanup}
For all $\rho \geq 0$, $\mu>0$, $0<\eps<\alpha \leq 1$ there exists a $\gamma>0$ such that for every 
sequence of densities $p=p(n)$ there exists an $n_0$ such that for any $n \geq n_0$ 
the following holds.

Let $\Gamma$ be a $(p,\beta)$-jumbled $n$-vertex graph with $\beta=\beta(n) \leq \gamma p^{1+\rho}n$ and let  $G\subset \Gamma$ be a subgraph of $\Gamma$ 
satisfying  $n' = |V(G)|\geq \mu n$ and $e(G)\geq \alpha p\binom{n'}2$. Then, there exists a subset $U\subset V(G)$ of size at least $\eps n'/50$ such that one of the following holds:
\begin{enumerate}[label=\RMlabel]
\item  \label{item:cleanup1} $e(G[U])\geq\left (\alpha+\eps^2/125\right)p\binom{|U|}2$ or
\item \label{item:cleanup2} $\deg_{G[U]}(u)= (\alpha\pm\eps)p|U|$ for all $u\in U$.
\end{enumerate}
\end{lemma}

To prove Lemma~\ref{lem:cleanup}, we shall make use of the following property immediately deduced from jumbledness. 

\begin{fact}
\label{lem:bij_size}
Let $\Gamma$ be a $(p,\beta)$-jumbled graph. If $X$, $Y \subseteq V(\Gamma)$ are 
disjoint and they satisfy $e_{\Gamma}(X,Y)\geq k|X|\neq p|Y||X|$, then 
\begin{equation}\label{eq:bij_size}
|X|\leq \frac{\beta^2|Y|}{(k-p|Y|)^2}.
\end{equation}
\end{fact}

\begin{proof}[Proof of Lemma~\ref{lem:cleanup}]
Given $\rho$,  $\mu$, $\eps$, and $\alpha$, set 
\begin{equation}\label{eq:constants2}
t^* = \frac{1}{2\rho}+1\,,\qquad\xi= \frac{8\eps^2}{25^2}\,,\qqand \gamma=\frac{\mu^2\eps^2\xi}{2^{4t^*+1}}\,.
\end{equation}
Given $p=p(n)$, let $n_0$  be sufficiently large, and let $\Gamma$ and $G$ be as stated. 
We suppose that~\ref{item:cleanup1} is not satisfied and show that~\ref{item:cleanup2} holds. 

We begin by passing to a large subgraph with a lower bound on its minimum degree. More precisely, we show that  
\begin{equation}\label{eq:U1}
\mbox{there exists a $W\subset V(G)$ with $|W|\geq \sqrt{\frac{\xi}{2}}n'$ and $\delta(G[W]) \geq (\alpha-\xi)p|W|$,}
\end{equation}
where $\delta(J)$ denotes the minimum degree of the graph $J$. 

Let $\{V_{n'},V_{n'-1}, \dots,V_{m}\}$ be a maximum sequence of vertex sets where $V_{n'} = V(G)$  and  let
 $V_{t-1}$ be obtained from $V_t$ by removing a  vertex in $V_t$  of degree less than 
 \[(\alpha-\xi)p|V_t| = (\alpha-\xi)pt\,.\] 
Given such a sequence, we set $W = V_{m}$. It remains to show that $|W|=m \geq \left(\xi/2\right)^{1/2}n'$. Indeed,
$$
\alpha p\binom{n'}2\leq e(G) \leq (\alpha-\xi)p\left(\sum_{t\in[n']}t\right)+e_{\Gamma}(W);
$$
and since $e_{\Gamma}(W) \leq p \binom{|W|}2 + \beta|W|$, we attain that 
$$
\alpha  \binom{n'}2  \leq (\alpha -\xi) \binom{n'+1}2 +  \binom{|W|}2 + \gamma p^{\rho}n |W|.
$$
As $n' \geq \mu n$, then $\gamma p^{\rho}n |W| \leq \frac{\gamma}{\mu} (n')^2$. 
Hence, for $n'$ sufficiently large (i.e., $n_0$ sufficiently large) we have
$$
\frac{\xi}{3} (n')^2 \leq \frac{|W|^2}{2} + \frac{\gamma}{\mu}(n')^2.
$$ 
Isolating $|W|^2$ then yields 
$$
2 \left( \frac{\xi}{3} - \frac{\gamma}{\mu} \right)(n')^2 \leq |W|^2. 
$$
Assertion \raf{eq:U1} now follows as $\gamma \leq \xi \mu /12$.                              

To handle the maximum degree we will repeatedly delete vertices with too high degree 
and show that this hardly effects the size of $W$ and the 
minimum degree condition obtained above.
We make this precise now.

Recall that $m=|W|$, let 
$$
X_1 =\big\{u\in W\colon \deg_{G[W]}(u)\geq (\alpha + 4\eps/5)p m\big\},
$$
and note that 
\begin{equation}\label{eq:X1}
|X_1| \leq \frac{\eps}{25}m.
\end{equation}
Indeed, $e_G(W) \leq \left( \alpha +\eps^2/125\right) p \binom{m}2$ as  assumption (I) does not hold. Then,
\begin{align*}
\left( \alpha +\frac{\eps^2}{125} \right) pm^2 
          &\geq 2 |E_G(W)| = \sum_{u \in W \setminus X_1} \deg_{G[W]}(u) +\sum_{x \in X_1} \deg_{G[W]}(x) \\
       \; & \geq (\alpha-\xi)pm(m-|X_1|) + \left(\alpha+\frac{4\eps}5\right)pm|X_1| \\ 
       \; & = (\alpha -\xi)pm^2+\left(\frac{4\eps}5+\xi\right)pm |X_1|;
\end{align*}
so that 
$$
|X_1| \leq \left( \frac{\eps^2/125+\xi}{4\eps/5 +\xi}\right) m \overset{\raf{eq:constants2}}{\leq} \frac{\eps}{25} m.
$$ 
 
Hence, by removing the set $X_1$ we have a guaranteed upper bound on the maximum degree.
However, the minimum degree of $G[W\setminus X_1]$ may now be lower than the bound requested in (II)
and in the remainder of the proof we focus on handling this issue.
   
To this end, we define a sequence of sets $(X_2, X_3,\dots)$. The set of vertices 
$X_2 \subseteq W \setminus X_1$ consists of those vertices whose degree in $G[W]$ into $X_1$ is ``too high''. 
In a similar manner, for $t>2$, we define
$X_t$ to be the subset of the remaining vertices having ``too high'' degree into $X_{t-1}$. 
We will show that such a sequence has constant length and that $\sum_{t\geq 2} |X_t|$ is 
negligible so that 
discarding all these sets does not affect the degrees of the remaining vertices. 

For $t \geq 2$, put 
$$
X_t = \left\{u \in W\setminus \bigcup_{j=1}^{t-1}X_j : \deg_{G[W]}(u,X_{t-1}) \geq \frac{\eps}{2^{t+2}} pm \right\}.
$$
In what follows, we prove by induction on $i$ that 
\begin{equation}\label{eq:Xi}
|X_t| \leq \gamma^{t-1} p^{2(t-1)\rho} m, \; \mbox{for all}\;  2 \leq t \leq t^*.
\end{equation}
Indeed, $|X_2|$ satisfies,
\begin{align*}
|X_2| \overset{\raf{eq:bij_size}}{\leq} \frac{\beta^2 |X_1|} {\left(\frac{\eps}{8} pm -p |X_1|\right)^2 }
       \leq \frac{ \frac{\eps}{25}m \gamma^2 p^2 p^{2 \rho} n^2}{\left(\frac{\eps}{8} - \frac{\eps}{25} \right)^2 p^2m^2} 
      \leq \left( \frac{\gamma}{\mu \frac{\eps}{25}\sqrt{\frac{\xi}{2}} } \right)^2 p^{2\rho} m 
      \overset{\raf{eq:constants2}}{\leq} \gamma p^{2\rho} m,
\end{align*}
where the third inequality holds due to $n^2 \leq \frac{m^2}{\mu^2 \xi/2}$. Consequently, \raf{eq:Xi} holds for $t=2$. 

For $t\geq 2$ we conclude from  \raf{eq:bij_size} that
$$
|X_{t+1}| \leq \frac{\beta^2 |X_t|}{\left(\frac{\eps}{2^{t+3}}pm - p|X_t| \right)^2}
$$ 
Substituting $|X_t|$ with the inductive hypothesis yields
\begin{align*}
|X_{t+1}| \leq \frac{ \gamma^2 p^2 p^{2\rho} n^2 \gamma^{t-1} p^{2 (t-1)\rho} m }{ \left(\frac{\eps}{2^{t+3}} - \gamma^{t-1}p^{2 (t-1) \rho} \right)^2 p^2m^2} 
\leq \frac{\gamma^2 \gamma^{t-1}p^{ 2 t \rho} m }{ \left(\frac{\eps}{2^{t+3}} - \gamma^{t-1}p^{2(t-1)\rho} \right)^2  \mu^2 \frac{\xi}{2}}.
\end{align*}
As by the choice of $\gamma$, the inequality
$$
\frac{\eps}{2^{t+3}} - \gamma^{t-1}p^{2(t-1) \rho} \geq \frac{\eps}{2^{t+4}}
$$ 
holds for each $2\leq t \leq t^*$, we reach 
$$
|X_{t+1}| \leq \frac{\gamma^2\gamma^{t-1}}{\left( \mu \frac{\eps}{2^{t+4}}  \sqrt{\frac{\xi}{2}}\right)^2  } p^{2 t \rho} m \overset{\raf{eq:constants2}}{\leq} \gamma^t p^{2 t \rho} m.
$$
This concludes our proof of \raf{eq:Xi}. \\

Next, we show that the length of the sequence $(X_t)$ is constant. In particular, we show that this sequence has length at most $t^* = \frac{1}{2\rho}+1$. To see this, observe that $X_{t+1}$ is empty if 
$|X_t| < \frac{\eps}{2^{t+2}} pm$. By \raf{eq:Xi}, the latter is satisfied if 
\begin{align}\label{eq:stop1}
\gamma^{t-1} p^{2(t-1) \rho} m & < \frac{\eps}{2^{t+2}} pm.\\
\intertext{For $t^*$, \raf{eq:stop1} is satisfied provided}
\gamma^{t^*-1} & \leq \frac{\eps}{2^{t^*+2}}, \nonumber
\end{align}
which indeed holds due to the choice of $\gamma$. 

In the remainder of the proof we show that we may choose $U = W \setminus \bigcup_{t=1}^{t^*} X_t$, so that
$|U| \geq \frac{\eps}{50} n'$ and  $G[U]$ satisfies (II). 
Observe that by the choice of $\gamma$ we have \begin{align*}
|U| & = |W| - |X_1| - \sum_{t=2}^{t^*} |X_t| \geq |W|-|X_1| - (t^*-1)|X_2| \\
\; & \geq m - \frac{\eps}{25}m -\frac{1}{2\rho} \gamma p^{2\rho} m  \geq 
\left(1 - \frac{3\eps}{50}\right) m.
\end{align*}
In particular, 
$|U| \geq \frac{\eps}{50} n'$ holds, as required. 

It remains to verify that $G[U]$ satisfies (II).  
We begin with the maximum degree of a vertex $v \in U$. Such a vertex satisfies 
$$
\deg_{G[U]}(v) < (\alpha + 4\eps/5) pm \leq (\alpha +4\eps/5)p \frac{|U|}{\left(1-\frac{3\eps}{50}\right)} 
\leq (\alpha+\eps)p|U|.
$$

Finally, we consider the minimum degree of a vertex $v \in U$.
Note that the choice of $\gamma$ and $t^*$ and \raf{eq:Xi} yield  $|X_{t^*}| \leq \frac{\eps}{50} pm$. Hence, we obtain
\begin{align*}
\deg_{G[U]}(v) \geq & \left(\alpha-\xi\right)pm -|X_{t^*}| -\frac{\eps}{4} pm 
\left(  \sum_{t=2}^{t^*} \frac{1}{2^t} \right)\\\geq& \left(\alpha-\xi - \frac{\eps}{50} - \frac{\eps}{4}\right)pm\geq(\alpha-\eps)p|U|
\end{align*}
which concludes the proof of Lemma~\ref{lem:cleanup}.
\end{proof}

Lemma~\ref{lem:partition} follows from Lemma~\ref{lem:cleanup} and a standard concentration result for the hypergeometric distribution. 

\begin{proof}[Proof of Lemma~\ref{lem:partition}]
Given $\l,\rho,\eps$, and $\alpha_0$, set 
\begin{equation*}
\eps_1 = \eps/4\,,\qquad\mu=\left(\eps_1/50\right)^{100 \eps_1^{-2}}\,,\qqand\nu= \mu/\l,
\end{equation*}
let $\gamma'$ be that obtained by applying Lemma~\ref{lem:cleanup} with $\rho,\eps_1$, and $\alpha_0$, and 
put
\begin{equation*}
\gamma = \min \{\gamma',\mu\}.
\end{equation*}
Given $p=p(n)$, let $n_0$ be sufficiently large as to accommodate Lemma~\ref{lem:cleanup}. 
For $n \geq n_0$, let~$\Gamma$ be an $n$-vertex $(p,\beta)$-jumbled graph, where $\beta$ is as specified in Lemma~\ref{lem:partition}, and let $G \subseteq \Gamma$ be a subgraph of $\Gamma$ satisfying $e(G)\geq \alpha_0 p\binom{n}2$.

To prove Lemma~\ref{lem:partition}, we shall first pass to a subgraph of $G$ that is essentially degree regular and of order linear in $n$. We then show that a random equipartition of such a subgraph is highly likely to be an $(\alpha,p,\eps)$-degree-regular $C_{\ell}(\nu n)$-graph for some $\alpha \geq \alpha_0$.  

In what follows, we show that there exists an $\alpha \geq \alpha_0$ and a set $U \subseteq V(G)$ satisfying
\begin{enumerate}[label=\rmlabel]
\item  $|U| \geq \mu n$ and
\item  $\deg_{G[U]}(u) = (\alpha \pm \eps_1)p|U|$  for each $u \in U$.
\end{enumerate}
Roughly speaking, to prove this we shall repeatedly apply Lemma~\ref{lem:cleanup} starting from $G_0=G$ and nested subgraphs thereof until assertion~\ref{item:cleanup2} of that lemma holds. 
Due to jumbledness such an iteration gives rise to a sequence $(G_0,\ldots,G_t)$ of nested subgraphs of $G_0$ where $t$ is a constant. We shall then set $U = V(G_t)$. We now make this precise.   

Set $G_0 = G$. For $i >0$, let $G_i \subseteq G_{i-1}$ be the subgraph of $G_{i-1}$ obtained by assertion (I) of Lemma~\ref{lem:cleanup} so that
\begin{align*}
|V(G_i)| & \geq \frac{\eps_1}{50} |V(G_{i-1})|\\
\intertext{and}
e(G_i) & \geq \left( \alpha_0 +i \frac{(\eps_1)^2}{25} \right) p \binom{|V(G_i)|}2.
\end{align*}
Owing to the jumbledness and the assumption that $e(G_0) \geq \alpha_0p \binom{|V(G_0)|}2$, it holds that $|V(G_0)|\geq \mu n$. 
Consequently, a sequence of the form $(G_0,G_1,\ldots)$ exists and we let $(G_0,\ldots,G_t)$ denote 
a maximal such sequence. Jumbledness with $\beta\leq \gamma pn$ yields
$$
\left(\alpha_0 + t\frac{\eps_1^2}{25}\right) p \binom{\mu n}2 \leq p \binom{\mu n}2 + \beta \mu n, 
$$
so that $t \leq \frac{100 \gamma}{\mu \eps_1^2} \leq \frac{100}{\eps_1^2}$. 
The existence of $\alpha$ as required follows.  

In the remainder of the proof, we show that a random equipartition of $G[U]$ is highly likely to be an 
$(\alpha,p,\eps)$-degree-regular $C_{\ell}(\nu n)$-graph. Without loss of generality, we may assume that 
$|U| = \ell \nu n$ (and thus divisible by $\ell$) and that 
$\deg_{G[U]}(u) = (\alpha \pm 2 \eps_1) p |U|$ for each $u \in U$. 
Indeed, a set $U' \subseteq U$ with $|U'| = \l \nu n \geq \frac{\mu}{2} n $ can be obtained by removing at most 
$\ell-1$ vertices from $U$. A vertex $u \in U'$ satisfies 
$\deg_{G[U']} = (\alpha \pm 2 \eps') p |U'|$ since $\l \leq \eps' p |U|$ for $n$ (i.e., $n_0$) 
sufficiently large and $p\gg n^{-1}$.  

Let $U=U_1\dcup\dots\dcup U_{\l}$ be a random equipartition of $U$ consisting of $\l$ sets each of size~$\nu n$. 
Call a vertex $v\in U$ {\sl bad} if there exists an index $j\in[\l]$ such that $v\not\in U_j$ and 
$$\big|\deg(v,U_j)-\alpha p|U_j|\big|>\eps p|U_j|.$$
Next, for two distinct indices $i, j\in[\l]$ and a vertex $v\in U_i$, let 
$X=X_v^j=\deg(v, U_j)$ denote the degree of $v$ into $U_j$ in $G[U]$. The random variable $X$ is hypergeometrically distributed with mean 
$$\ex{X}=\frac{(\alpha\pm 2 \eps_1)p|U| |U_j|}{|U|} = (\alpha\pm 2 \eps_1)p|U_j|.$$
For hypergeometrically distributed random variable the following is a well-known concentration result (see e.g.,~\cite{JLR}*{Theorem 2.10 and Equation (2.9)})
\begin{align*}
\pr{|X-\ex{X}|\geq \eta \ex{X}}\leq 2\exp(-\eta\ex{X}/3) \quad \text{ for } \eta\leq 3/2. 
\end{align*}
Consequently, the probability that a fixed vertex $v$ is bad is given by 
\begin{align*}
\pr{\mbox{$v$ is bad}} & \leq \l \pr{|X- \ex{X}|> (\eps-2\eps_1)p \nu n}
                                        \leq \l \pr{|X - \ex{X} | > \frac{\eps-2\eps_1}{\alpha+2\eps_1} \ex{X}} \\
                               \;   & \leq 2 \l  \exp\left(- \frac{\eps - 2\eps_1}{3(\alpha+2\eps_1)} \ex{X}\right).
\end{align*}
Hence, for $n$ sufficiently large since $p \gg \ln n /n$ we have that 
\begin{align*}
\pr{\mbox{there exists a bad vertex in $U$}} \leq 2 \l |U|  \exp \left(- \frac{\eps - 2\eps_1}{3(\alpha+2\eps_1)} \ex{X}\right) <1,
\end{align*}
implying that there exists an equipartition of $U$ yielding no bad vertices. Such a partition, with possible redundant edges removed,  
forms an $(\alpha,p,\eps)$-degree-regular $C_{\l}(\nu n)$-graph,
as required.
\end{proof}

\section{Proof of Lemma~\ref{lem:main}}
\label{sec:mainlemma}

In this section we prove Lemma~\ref{lem:main}. Throughout this section, $\Gamma$ denotes an $n$-vertex $(p,\beta)$-jumbled graph and $H$ denotes an $(\alpha,p,\eps)$-degree-regular $C_{2k+1}(m)$-graph that is a subgraph of $\Gamma$. We assume that the graph $H$ has a partition $(U_k,\dots,U_1,W,V_1,\dots,V_{k})$ of its vertex set (see Section~\ref{sec:proof-main}). 

Lemma~\ref{lem:main} has two parts the first of which concerns $\cC(H,\Gamma)$. Recall that this set consists of all $(2k+1)$-cycles of the form $(u_k,\dots,u_1,w,v_1,\dots,v_k)$, where $u_i\in U_i$, $w\in W$, $v_i\in V_i$, the edge $u_k v_k$ is an edge of $\Gamma[U_k,V_k]$, and the remaining edges are that of $H$.
The  first part of the lemma (see \raf{eq:wishfor1}) asserts that 
\[
|\cC(H,\Gamma)| \geq \left(\alpha-2\eps\right)^{2k}(pm)^{2k+1}\,.
\]
Observe that the fact that $H$ is almost degree regular implies that the number of paths of 
the form $(u_k,\dots,u_1,w,v_1,\dots,v_k)$ in $H$ is  $((\alpha\pm\eps) pm)^{2k}$. As $H$ is an arbitrary subgraph of $\Gamma$ it may occur that this set of paths ``clusters'' on a small number of pairs of vertices $(u_k,v_k) \in U_k \times V_k$. The lower bound stated above 
asserts that this is not the case. In fact, approximately a $p$ proportion of these paths extend to cycles in $\cC(H,\Gamma)$ as one would expect in a purely random setting. 

The second part of  Lemma~\ref{lem:main} concerns the set $\cO(\alpha+2\eps,H,\Gamma)$ of $(\alpha+2\eps)$-saturated cycles. This set consists of all cycles $(u_k,\dots,u_1,w,v_1,\dots,v_k)$ in 
 $\cC(H,\Gamma)$ for which the edge $u_k v_k \in E(\Gamma[U_k,V_k])$ is $(\alpha+2\eps)$-saturated, meaning that it is contained in at least $p((\alpha+2\eps) pm)^{2k-1}$ members of $\cC(H,\Gamma)$. In \raf{eq:wishfor2}, the second part of the lemma, an upper bound is put forth for $|\cO(\alpha+2\eps,H,\Gamma)|$ asserting that the latter is negligible compared to~$|\cC(H,\Gamma)|$. The point here is that for $|\cC(\Gamma,H)|\geq(\alpha-2\eps)^{2k} (pm)^{2k+1}$ (\raf{eq:wishfor1} will yield this) and $e_{\Gamma}[U_k,V_k]$ is approximately $pm^2$, we expect an edge of $\Gamma[U_k,V_k]$ to be contained in roughly at least $p(\alpha-2\eps)^{2k}(pm)^{2k-1}$ members of $\cC(H,\Gamma)$. 
 Hence, for a small $\eps$, the number of cycles in  $\cC(H,\Gamma)$ an $(\alpha+2\eps)$-saturated 
 edge is contained in
 overshoots this expectated lower bound by a factor of roughly $1/\alpha$. 
 The second part of Lemma~\ref{lem:main} asserts that the contribution of these saturated edges is negligible compared
 to $|\cC(H,\Gamma)|$

In the sequel we introduce a framework that will enable us to handle estimates for both~$|\cC(H,\Gamma)|$ and $|\cO(\mu,H,\Gamma)|$ using essentially the same type of arguments. Prior to this framework, we include here a brief and rough sketch of our approach for $k=3$ in which~$H$ has the partition $(U_3,U_2,U_1,W,V_1,V_2,V_3)$.

For a vertex $w \in W$, we write $\cC(H,\Gamma,w)$ to denote the cycles in $\cC(H,\Gamma)$ that contains $w$. Clearly, an estimate for such a contribution will translate into an estimate for $|\cC(H,\Gamma)|$. To estimate $\cC(H,\Gamma,w)$,
 we shall repeatedly apply the jumbledness condition \eqref{eq:jumbled1}  to pairs of subsets $(L,R)$ where $L \subseteq U_3$ and $R \subseteq V_3$ consist 
 of certain vertices connected to~$w$ by a path of length three. 
 The framework to be introduced will ensure us that in each such application of the jumbledness condition to a 
 pair $(L,R)$, the edges in $\Gamma[L,R] \subseteq \Gamma[U_k,V_k]$ are contained in 
 roughly the same number of cycles in $\cC(H,\Gamma,w)$. 

To achieve this level of control, consider the neighborhood $X_w=N_H(w)\cap U_1$ of $w$ in~$U_1$. 
We partition the set $U_2$ according to the ``backward'' degrees of its vertices into~$X_w$. 
More precisely, the $i$th partition class will consist of all vertices $u\in U_2$ satisfying
\[(1+\eta)^{i-1}\leq |N_H(u)\cap X_w|<(1+\eta)^{i}\] 
for some small $\eta$ to be chosen later on.
Some of these classes may be empty. Nevertheless, these classes cover $U_2$ up to the vertices not connected to 
$X_w$. Moreover, there are  at most $\lceil \log_{1+\eta} n\rceil+1$ such classes.

We proceed to $U_3$ in a similar manner. Each partition class of $U_2$ will define a partition of~$U_3$. The latter is defined in a similar manner to the partition just defined for $U_2$ using $X_w$.  More precisely,  given the $i$-th partition class of $U_2$, say,  $Z_{\eta}(i,X_w)$, we assign a vertex $u\in U_3$ to the $j$-th partition class of $U_3$ if it satisfies $(1+\eta)^{j-1}\leq |N_H(u)\cap Z_{\eta}(i,X_w)|<(1+\eta)^{j}$. The resulting partition class is denoted    
$Z_{\eta}(i,j,X_w)$.

We partition the sets $V_2$ and $V_3$ in a similar manner where we use $Y_w=N_H(w)\cap V_1$, the neighborhood of $w$ in $V_1$, instead  of $X_w$. 

The advantage of this kind of partitioning is that the number of paths  between $w$ and any vertex $u\in Z_{\eta}(i,j,X_w)\subset U_3$, confined to the set $Z_{\eta}(i,X_w)$ -
i.e.\ paths of the form   $(w,u',u)$ where $u' \in Z_{\eta}(i,X_w)$ -
is at least $(1+\eta)^{i+j-2}$ and at most $(1+\eta)^{i+j}$.  As a result, this number is known up to a factor of $(1+\eta)^2$. Naturally, the same bounds hold for~$w$ and any  
vertex $v\in Z_{\eta}(i',j',Y_w)\subset V_3$, where $Z_{\eta}(i',j',Y_w)$ is the set
obtained by the  partitioning procedure with respect to $Y_w$, $V_2$, and $V_3$. 

Now, if we take a path from $w$ to $u\in Z_{\eta}(i,j,X_w)$, confined to $Z_{\eta}(i,X_w)$, and another path from $w$ to $v\in Z_{\eta}(i',j',Y_w)$, confined to $Z_{\eta}(i',X_w)$, then these two paths yield a path from $u$ to $v$. Hence, the number of such $(u,v)$-paths is at least 
$(1+\eta)^{i+i'+j+j'-4}$ and at most $(1+\eta)^{i+i'+j+j'}$.

Since $u$ and $v$ were arbitrary vertices from $Z_{\eta}(i,j,X_w)$ and $Z_{\eta}(i',j',Y_w)$, respectively, we conclude that
if $uv$ is an edge of $\Gamma$, then $(1+\eta)^{i+i'+j+j' \pm4}$ is  the number of cycles in~$\cC(H,\Gamma,w)$ that contain 
the edge $uv$ and are confined to $Z_{\eta}(i,X_w)$ and $Z_{\eta}(i',X_w)$.

On the other hand, jumbledness of $\Gamma$ yields that the number of edges in $\Gamma$ spanned between~$Z_{\eta}(i,j,X_w)$ and $Z_{\eta}(i',j',Y_w)$ is  
\begin{equation*}
	p|Z_{\eta}(i,j,X_w)||Z_{\eta}(i',j',Y_w)|\pm\beta\sqrt{|Z_{\eta}(i,j,X_w)||Z_{\eta}(i',j',Y_w)|}.\end{equation*}
Summing over all $1 \leq i,j,i',j'\leq \lceil\log_{1+\eta}n\rceil +1$, we obtain a good estimate for  $|\cC(H,\Gamma,w)|$.
Indeed, the main term of the contribution of $w$ is
\[\sum_{i,j,i',j'}p|Z_{\eta}(i,j,X_w)||Z_{\eta}(i',j',Y_w)|(1+\eta)^{i+j+i'+j'\pm4}.\]
This, we will show to be  
\begin{equation}\label{eq:example-main}
|X_w||Y_w|p((\alpha\pm\eps)pm)^{4},
\end{equation}
by a simple argument.
Note that we also obtain an upper bound here which will turn out important later.
The main obstacle will be to show that the error term  
is negligible compared to the main term, i.e., that 
\begin{equation}\label{eq:example-error}\sum_{i,j,i',j'}\beta\sqrt{|Z_{\eta}(i,j,X_w)||Z_{\eta}(i',j',Y_w)|}(1+\eta)^{i+j+i'+j'+4}=o(p(pm)^{6}),\end{equation}
 provided $\Gamma$ is sufficiently jumbled. This will be done in Claim~\ref{lem:qxqy}. 

So far we have discussed our approach for establishing the first part of Lemma~\ref{lem:main}, i.e.,~\raf{eq:wishfor1}. For the second part of this lemma, i.e.,~\raf{eq:wishfor2}, we are to estimate of $|\cO(\alpha+2\eps,H,\Gamma)|$. 
This will be done by employing similar arguments to those above that will be applied to a rearrangement of the partition 
$(U_3,U_2,U_1,W,V_1,V_2,V_3)$. In particular we shall use the rearrangement
 \[(\tilde U_3,\tilde U_2,\tilde U_1,\tilde W,\tilde V_1,\tilde V_2,\tilde V_{3})=(W,U_1,U_2,U_3,V_3,V_2,V_1).\]
This is a valid partition of the $C_7(m)$-graph $H$. 

The interest here is to estimate the number of $(\alpha+2\eps)$-saturated cycles. The $(\alpha+2\eps)$-saturated edges that these cycles contain now lie between the sets
$\tilde W=U_3$ and $\tilde V_1=V_3$. For a given vertex $\tilde w\in \tilde W$ we set, as before, $\tilde X_{\tilde w}=N_H(\tilde w)\cap \tilde U_1$. 
Unlike before, we shall define the set  $\tilde Y_{\tilde w}\subset \tilde V_1$ to consist of those vertices of $\tilde V_1$ that are incident to $\tilde w$ through $(\alpha+2\eps)$-saturated edges.
 
The same arguments as above, yield bounds corresponding to \eqref{eq:example-main} and \eqref{eq:example-error} that will then lead to an upper bound on the number of $(\alpha+2\eps)$-saturated cycles containing $\tilde w$. 
In particular, we shall have that since every $(\alpha+2\eps)$-saturated edge is contained in at least $p ((\alpha+2\eps)pm)^{5}$ cycles containing $\tilde w$, then
\begin{equation}\label{ex:saturated}
|\tilde Y_{\tilde w}| p ((\alpha+2\eps)pm)^{5}\leq |\tilde X_w||\tilde Y_w|p((\alpha+\eps)pm)^{4}+o(p(pm)^{6}).
\end{equation}
Now, as $|\tilde X_w|\leq(\alpha+\eps)pm$, we conclude that this last inequality can hold provided that $|\tilde Y_{\tilde w}|=o(pm)$, implying that the number of $(\alpha+2\eps)$-saturated cycles containing $\tilde w$ is  bounded from above by the right hand side of \eqref{ex:saturated}, which is $o(p(pm)^6)$. Summing over all~$\tilde w\in \tilde W$ yields the desired bound.

\subsection{Preparation for the proof of Lemma~\ref{lem:main}} 
In what follows we make the above discussion precise and make it fit for a general $k$. 
As already mentioned, here $\Gamma$ is a $(p,\beta)$-jumbled graph and $H$ is a subgraph of $\Gamma$ that is 
$(\alpha,\eps,p)$-degree-regular $C_{2k+1}(m)$-graph. The latter we assume to have the partition $(U_k,\dots,U_1,W,V_1,\dots,V_{k})$ of its vertex set.

\subsubsection*{Partitioning the neighborhoods} 
For a real $\eta >0$, set $L_{\eta} = \lceil \log_{1+\eta} 2pm \rceil+1$,
and let
$$
\cI_{\eta} = \{0\} \times [L_{\eta}]^{k-1},
$$
where the use of zero here will be made clear shortly. 
By $\vek{s}$ we mean a tuple of integers $(s_1,\dots,s_k) \in \cI_{\eta}$ (so that $s_1$ is always zero), and
write $\vek s_j$ to denote the prefix $(s_1,\ldots,s_j)$, where $j \in [k]$, and write $\vek s$ instead of $\vek s_{k}$.
For a set $X\subset U$ we put $Z_{\eta}(\vek s_1,X) = Z_{\eta}(0,X) =X\subset U_1$, and for $j=2,\dots,k$ we define 
\begin{equation}\label{eq:defXY}
Z_{\eta}(\vek s_{j},X) =  \{x\in U_j \colon (1+\eta)^{s_{j}-1}\leq |N_H(x)\cap Z_{\eta}(\vek{s}_{j-1},X)| <(1+\eta)^{s_j}\},
\end{equation}
so that $Z_{\eta}(\vek s_j,X) \subseteq U_j$ for each $j \in [k]$.
For future reference, it will be convenient for us to stress that for $\eta\in (0,1]$, the value  $(1+\eta)^{s_j}$ is essentially bounded by the maximum degree of $H$ for any $j \in [k]$, in particular, it holds that 
\begin{equation}\label{eq:def-deg}
\mbox{if $\eta \in (0,1]$, then $(1+\eta)^{s_j} \leq 8pm$, for any $s_j \in [L_{\eta}]$;}  
\end{equation}
indeed, 
$$
L_{\eta} \leq \log_{1+\eta}(2pm) +2 \leq \log_{1+\eta} (2pm) + \log_{1+\eta} 4 = \log_{1+\eta}8pm.
$$ 
Observe, in addition, that $(1+\eta)^{L_{\eta}-1} \geq 2pm \geq (\alpha+\eps)pm$, and that the maximum degree of a vertex in $H$ is $(\alpha+\eps)pm$. This means that \raf{eq:defXY} defines a partition of the neighborhood of $Z_{\eta}(\vek s_{j-1},X)$ in $U_j$, i.e.,
\begin{equation}\label{eq:set-partition}
\dot\bigcup_{1\leq i \leq L_{\eta}} Z_{\eta}((s_1,\ldots,s_{j-1},i),X) = N_H(Z_{\eta}(\vek s_{j-1},X)) \cap U_j,
\end{equation}
where some of these sets may possibly be empty.

\subsubsection*{Counting paths} 
We exploit the above partitioning scheme in order to count $(X,U_k)$-paths in $H$. Throughout,  by {\sl paths} we always mean {\sl shortest paths}. Also, if $L$ and $R$ are two subsets of vertices, we write $(L,R)$-path to denote a path with one end in $L$ and the other in $R$. With these conventions, an $(X,U_k)$-path in $H$, where $X\subset U_1$, has a single vertex in each set $U_i$, $i\in[k]$. Finally, instead of $(X,\{y\})$-{\sl path} we write $(X,y)$-{\sl path}.

For $j\in \{2,\ldots,k\}$ and a tuple $\vek s \in \cI_{\eta}$, we write $\sum \vek s_j$ to denote the sum 
$\sum_{i=1}^{j} s_i = \sum_{i=2}^j s_i$, and we write $\sum \vek s$ instead of $\sum \vek s_{k}$. 
Further,  the subgraph of $H$ induced by the vertex sets $\{Z_{\eta}(\vek s_1,X), Z_{\eta}(\vek s_2,X) \ldots, Z_{\eta}(\vek s_{j},X)\}$ is denoted by $H(\vek s_j)$ (we suppress the dependence on~$\eta$ here).

For a vertex $z\in Z_{\eta}(\vek s_{j},X)$ the number $\pi_H(X,\vek s_j,z)$ of $(X,z)$-paths confined to  $H(\vek s_j)$ clearly satisfies
$$
\prod_{i=2}^{j}(1+\eta)^{s_i-1} \leq \pi_H(X,\vek s_j,z) \leq \prod_{i=2}^{j}(1+\eta)^{s_i}.
$$
Since  $s_1 = 0$, we may write 
\begin{equation}\label{eq:pathsintube}
(1+\eta)^{-(j-1)}(1+\eta)^{\sum \vek s_j } \leq \pi_H(X,\vek s_j,z) \leq (1+\eta)^{\sum \vek s_j }.
\end{equation}
Recall the benefit of the above partitioning scheme, we observe that for any two vertices $z,z' \in Z_{\eta}(\vek s_j,X)$, the variation between $\pi_H(X,\vek s_j,z)$ and $\pi_H(X,\vek s_j,z')$ is bounded by a factor of $(1+\eta)^{j-1}$. 
Hence, the number 
\[
\pi_H(X ,\vek s_j)=\sum_{z\in Z_{\eta}(\vek s_{j},X)}\pi_H(X,\vek s_j,z)
\]
of $(X ,Z_{\eta}(\vek s_j,X))$-paths confined to $H(\vek s_j)$ satisfies
\[(1+\eta)^{-(j-1)}|Z_{\eta}(\vek s_j,X)|(1+\eta)^{\sum \vek s_j } \leq \pi_H(X,\vek s_j) \leq |Z_{\eta}(\vek s_j,X)|(1+\eta)^{\sum \vek s_j }.\]

By \eqref{eq:set-partition}, every $(X,U_j)$-path is contained in $H(\vek s_j)$ for exactly one $\vek s_j$. Hence, summing over all $\vek s_j$, we obtain the following inequality for the number $\pi_H(X,U_j)$ of $(X,U_j)$-paths:
\begin{equation}\label{eq:boundpx}
(1+\eta)^{-(j-1)}\sum_{\vek s_j } |Z_{\eta}(\vek s_j,X)|(1+\eta)^{\sum \vek s_j}  
\leq \pi_H(U_j,X) \leq \sum_{\vek s_j }|Z_{\eta}(\vek s_j,X)|(1+\eta)^{\sum \vek s_j}\,.
\end{equation}
On the other hand, owing to the degree-regularity of $H$, we obviously have
\begin{equation}\label{eq:path-limits}
|X| \left((\alpha-\eps)pm \right)^{j-1} \leq \pi_H(U_j,X) \leq |X| \left((\alpha+\eps)pm \right)^{j-1}
\end{equation}
for all $j\in[k]$.

We conclude this section by mentioning that for a set $Y \subseteq V_1$ and a tuple $\vek t \in \cI_{\eta}$, we define the sets $\{Z_{\eta}(\vek t_j,Y)\}_{j=1}^{k}$ 
and the numbers $\pi_H(Y,\vek t_j,u)$, $\pi_H(Y,\vek t_j)$,  $\pi_H(V_j,Y)$ in an analogous manner to the sets and numbers just defined. 
The properties \raf{eq:set-partition} to \raf{eq:path-limits} translate verbatim.

\subsubsection*{Counting cycles} 
Given $u\in U_k$, $v\in V_k$, $\pi_H(X,u)$, and  $\pi_H(Y,v)$ (the number of $(X,u)$-paths and $(Y,v)$-paths, respectively), we have that 
\begin{equation}
\cS(X,Y) =\sum_{uv\in E_{\Gamma}(U_k,V_k)} \pi_H(X,u)\pi_H(Y,v)\label{eq:SXY-con}
\end{equation}
is the number of composed paths each of which comprises a $(X,u)$-path and a $(Y,v)$-path (in $H$) connected 
by the edge $uv \in \Gamma[U_k,V_k]$.
 Let $\cC(H,\Gamma,w)$ denote the set of cycles in~$\cC(H,\Gamma)$ containing the vertex $w\in W$, and observe that
\begin{equation}
|\cC(H,\Gamma,w)|  =  \cS(N_H(w) \cap U_1,N_H(w) \cap V_1) \label{eq:single-SXY}.
\end{equation}

Next, let us consider $|\cO(\mu,H,\Gamma)|$.  We
rearrange the partition $(U_k,\dots,U_1,W,V_1,\dots,V_{k})$ to yield a partition 
$(\tilde U_k,\dots,\tilde U_1,\tilde W,\tilde V_1,\dots,\tilde V_{k})$ obtained by renaming the partition classes 
as follows: 
\begin{equation}\label{eq:rename}
\tilde W = U_k, \; \tilde U_1 = U_{k-1},\ldots, \tilde U_{k-1} = U_1,\; \tilde U_k = W, \; \tilde V_1 = V_k,\ldots, \tilde V_k = V_1. 
\end{equation}
Note that  the new partition is still a valid partition of the $C_{2k+1}(m)$-graph $H$, and that the 
$\mu$-saturated edges now lie between $\tilde W$ and $\tilde V_1$. For a vertex $\tilde w\in \tilde W$, let
$D_{\mu}(\tilde w)$ denote the set of vertices in $\tilde V_1$ adjacent to  $\tilde w$ (in $\Gamma$) through a $\mu$-saturated edge and
let  $\cO(\mu,H,\Gamma,\tilde w)$ denote the set of $\mu$-saturated cycles containing $\tilde w$. Then,
\begin{equation}
|\cO(\mu,H,\Gamma,\tilde w)|  \leq  \tilde \cS(N_H(\tilde w) \cap \tilde U_1, D_{\mu}(\tilde w)), \label{eq:saturated-connection}
\end{equation}
where $\tilde \cS(\tilde X, \tilde Y)$ is defined  in the same way as $\cS(\tilde X, \tilde Y)$ only with respect to the partition
$(\tilde U_k,\dots,\tilde U_1,\tilde W,\tilde V_1,\dots,\tilde V_{k})$, and where $\tilde X\subset\tilde U_1$ and $\tilde Y\subset \tilde V_1$.
In \raf{eq:saturated-connection}, we obtain an upper bound only as cycles in $\tilde \cS(N_H(\tilde w) \cap \tilde U_1, D_{\mu}(\tilde w))$ may involve
edges in $\Gamma$ between $\tilde U_k$ and~$\tilde V_k$ which might not belong to $H$.

In view of \raf{eq:single-SXY} and \raf{eq:saturated-connection}, we focus on $\cS(X,Y)$ in order to estimate $|\cC(H,\Gamma,w)|$ and $|\cO(\mu,H,\Gamma,\tilde w)| $. 
For tuples $\vek s, \vek t \in \cI_{\eta}$, we write $e_{\Gamma}(\vek s,\vek t)$ for $e_{\Gamma}(Z_{\eta}(\vek s,X),Z_{\eta}(\vek t,Y))$, and observe that due to 
\raf{eq:boundpx}  we may write
\begin{equation}\label{eq:SXY}
\frac{ \sum_{\vek{s}}\sum_{\vek t} e_{\Gamma}(\vek s,\vek t)(1+\eta)^{\sum \vek s+\sum \vek t}}{(1+\eta)^{2(k-1)}} \leq \cS(X,Y) \leq
\sum_{\vek{s}}\sum_{\vek t} e_{\Gamma}(\vek s,\vek t)(1+\eta)^{\sum \vek s+\sum \vek t}.
\end{equation}
Here, we appeal to the jumbledness of $\Gamma$ to estimate $e_{\Gamma}(\vek s,\vek t)$ which asserts that 
\begin{equation}\label{eq:est}
e_{\Gamma}(\vek s,\vek t)=p|Z_{\eta}(\vek s,X)||Z_{\eta}(\vek t,Y)|\pm \beta \sqrt{|Z_{\eta}(\vek s,X)||Z_{\eta}(\vek t,Y)|}.
\end{equation}
Substituting this estimate for $e_{\Gamma}(\vek s,\vek t)$  in \raf{eq:SXY}, we arrive at the following two bounds for~$\cS(X,Y)$.
\begin{align}
\label{eq:lowerboundcS}  
(1+\eta)^{2(k-1)}\cS(X,Y) & \geq p P_{\eta}(X) P_{\eta}(Y) - \beta Q_{\eta}(X)Q_{\eta}(Y), \\
\intertext{and}
\label{eq:upperboundcS}
\cS(X,Y) & \leq pP_{\eta}(X)P_{\eta}(Y)+\beta Q_{\eta}(X)Q_{\eta}(Y),
\end{align}
where
\begin{align}\label{eq:PundQ}
P_{\eta}(X)=\sum_{\vek s}|Z_{\eta}(\vek s,X)|(1+\eta)^{\sum \vek s}
\text{ and } 
Q_{\eta}(X)=\sum_{\vek s}\sqrt{|Z_{\eta}(\vek s,X)|}(1+\eta)^{\sum \vek s}.
\end{align}

Hence, for a small $\eta$, the size of $\cS(X,Y)$ is essentially determined up to the additive error term $\beta Q_{\eta}(X)Q_{\eta}(Y)$. In the sequel, we show that this is dominated by the main term $pP_{\eta}(X)P_{\eta}(Y)$. 
The estimate of the error term is complicated and consequently delegated to Claim~\ref{lem:qxqy} below. The estimate of the main term, however, is almost trivial at this point. 
To see the latter, we rewrite \raf{eq:boundpx} for $j=k$ as to obtain   
$$
P_{\eta}(X) \leq \pi_H(U_k,X) (1+\eta)^{k-1} \leq P_{\eta}(X) (1+\eta)^{k-1}.
$$
With \raf{eq:path-limits} this yields 
\begin{align}\label{eq:boundsPX}
|X|\left((\alpha-\eps)pm \right)^{k-1} \leq P_{\eta}(X) \leq |X|\left((\alpha+\eps)(1+\eta)pm \right)^{k-1}. 
 \end{align}
A similar assertion clearly holds for $P_{\eta}(Y)$.

We may now summarize all of the above discussion concisely as follows. 
For $ w \in W$, the sets $X_w = N_H(w)\cap U_1$ and $Y_w = N_H(w)\cap V_1$ both have size $(\alpha \pm \eps)pm$ due to the degree regularity of $H$. Owing to 
\eqref{eq:single-SXY},  \raf{eq:lowerboundcS}, and \raf{eq:boundsPX} we have that
\begin{equation}\label{eq:single-C}
(1+\eta)^{2k}|\cC(H,\Gamma,w)|  \geq \left(\alpha-\eps \right)^{2k}p(pm)^{2k} -\beta Q_{\eta}(X_w)Q_{\eta}(Y_w).
\end{equation}

Next, let $\tilde w \in \tilde W$ and $X_{\tilde w} =N_H(\tilde w) \cap \tilde U_1$ and 
$Y_{\tilde w} = D_{\mu}(\tilde w)$. Moreover, let $\tilde P_{\eta}(\tilde X)$ and~$\tilde Q_{\eta}(\tilde X)$ be 
defined analogously to  $P_{\eta}(X)$ and $Q_{\eta}(X)$ as in \eqref{eq:PundQ}.

We have, owing to \raf{eq:saturated-connection},  \raf{eq:upperboundcS}, and \raf{eq:boundsPX}, that
\begin{align}
|\cO(\mu,H,\Gamma,\tilde w)| & \overset{\raf{eq:saturated-connection}, \raf{eq:upperboundcS}}{\leq} p \tilde P_{\eta}(X_{\tilde w}) \tilde P_{\eta}(D_{\mu}(\tilde w)) + \beta  \tilde Q_{\eta}(X_{\tilde w}) \tilde Q_{\eta}(D_{\mu}(\tilde w)) \nonumber \\
\; & \overset{\phantom{\raf{eq:saturated-connection}}}{\leq} |X_{\tilde w}||D_{\mu}(\tilde w)| p \left((\alpha+\eps)pm \right)^{2(k-1)}+ \beta  \tilde Q_{\eta}(X_{\tilde w}) \tilde Q_{\eta}(D_{\mu}(\tilde w)) \nonumber \\
\label{eq:single-O}
\; & \overset{\phantom{\raf{eq:saturated-connection}}}{\leq}  |D_{\mu}(\tilde w)| p \left( (\alpha+\eps)pm \right)^{2k-1} +\beta \tilde Q_{\eta}(X_{\tilde w}) \tilde Q_{\eta}(D_{\mu}(\tilde w)).
\end{align}

We conclude this section by stating the claim that will be used to control the error term $\beta Q_{\eta}(X)Q_{\eta}(Y)$ (and $\beta \tilde Q_{\eta}(\tilde X)\tilde Q_{\eta}(\tilde Y)$) discussed above.    

\begin{claim}
\label{lem:qxqy}
For any integer $k\geq 1$, and reals $0< \xi, \alpha,\eta,\nu \leq 1$, and $0< \eps \leq \alpha/3$, there exists a $\gamma>0$  such that for every sequence of densities $p=p(n)>0$ there exists an $n_0$ such that for all $n>n_0$ the following holds.
 
Let $\Gamma$ be an $n$-vertex $(p,\beta)$-jumbled graph with 
\begin{equation}\label{eq:cl8beta}
\beta < \gamma  p^{1+\frac1{2k-1}}n\log^{-2(k-1)}n\,,
\end{equation}
and let $H\subseteq \Gamma$ be an $(\alpha,p,\eps)$-degree-regular $C_{2k+1}(m)$-graph, 
$m \geq \nu n$, with the partition $(U_k,\dots,U_1,W,V_1,\dots,V_k)$. 
If $X\subset U_1$ and $Y\subset V_1$ both have size at most $(\alpha+\eps)pm$, then
$$
\beta Q_{\eta}(X)Q_{\eta}(Y)<\xi p(pm)^{2k}.
$$
\end{claim}

We postpone the proof of Claim~\ref{lem:qxqy} until Section~\ref{sec:qxqy}. In the subsequent section we show how to derive Lemma~\ref{lem:main} from Claim~\ref{lem:qxqy}.

\subsection{Proof of Lemma~\ref{lem:main}}
Given $k,\nu,\alpha_0$, and $\eps$, put
\begin{equation}\label{eq:choice-lem-main}
 \xi = \left(\eps/4\right)^{4k}\,,\qquad\eta = \min\left\{\frac{\eps}{4 \alpha_0},1\right\}, 
\end{equation}
and let $\gamma'$ be that obtained from Claim~\ref{lem:qxqy} applied with $k,\; \xi,\; \alpha = \alpha_0,\; \eta, \;\nu$, and $\eps$. For Lemma~\ref{lem:main} we set 
$$
\gamma = \gamma',
$$ 
and choose $n_0$ to be sufficiently large as to accommodate  Claim~\ref{lem:qxqy}.
Finally, let $\Gamma$ and $H$ be as specified in Lemma~\ref{lem:main}.

Owing to \raf{eq:single-C}
\begin{align*}
|\cC(H,\Gamma,w)| & \geq   \left(\frac{\alpha-\eps}{1+\eta} \right)^{2k}p(pm)^{2k} - \frac{1}{(1+\eta)^{2k}}\beta Q_{\eta}(X_w)Q_{\eta}(Y_w),\\
\intertext{for any $w \in W$, where $X_w = N_H(w)\cap U_1$ and $Y_w = N_H(w)\cap V_1$. 
Since $H$ is $(\alpha,p,\eps)$-degree-regular both $X_w$ and $Y_w$ have size at most $(\alpha+\eps)pm$. As a result, 
\[
	\beta Q_{\eta}(X_w)Q_{\eta}(Y_w) \leq \xi p(pm)^{2k}\,,
\] 
by Claim~\ref{lem:qxqy} applied with $X =X_w$ and $Y=Y_w$. Then, owing to our choices for $\eta$ and $\xi$ in~\raf{eq:choice-lem-main} we have that}
|\cC(H,\Gamma,w)| & \geq \left(\frac{\alpha-\eps}{1+\eta} \right)^{2k}p(pm)^{2k} - (1+\eta)^{-2k}\xi p(pm)^{2k}\\
\; & \geq \frac{(\alpha-\eps)^{2k} - (\eps/4)^{4k}}{(1+\eta)^{2k}}p(pm)^{2k} \\
\; & \geq \left( \frac{\alpha-(3/2)\eps}{1+\eta} \right)^{2k} p(pm)^{2k}\\
\; & \geq (\alpha-2\eps)^{2k}p(pm)^{2k}
\end{align*} 
holds for any $w \in W$. Summing over all vertices in $W$ yields~\raf{eq:wishfor1}, the first assertion of Lemma~\ref{lem:main}.  

It remains to show \raf{eq:wishfor2}, the second assertion of Lemma~\ref{lem:main}. Here, we use the partition of~\raf{eq:rename}. It is sufficient to prove that for any $\tilde w \in \tilde W$ it holds that 
\begin{equation}\label{eq:goal-D}
|D_{\alpha+2\eps}(\tilde w)|\leq \eps^{2k}pm.
\end{equation}
Indeed, assuming \raf{eq:goal-D} yields 
\begin{align*}
|\cO(\alpha+2\eps,H,\Gamma,\tilde w)| & \overset{\raf{eq:single-O}}{\leq} |D_{\alpha+2\eps}(\tilde w)| p \left( (\alpha+\eps)pm \right)^{2k-1} +\beta \tilde Q_{\eta}(X_{\tilde w}) \tilde Q_{\eta}(D_{\alpha+2\eps}(\tilde w))\\
\; & \overset{\phantom{\raf{eq:single-O}}}{\leq} \eps^{2k}(\alpha+\eps)^{2k-1}p (pm)^{2k} +\xi p (pm)^{2k} \; \; \mbox{(by Claim~\ref{lem:qxqy})}\\
\; & \overset{ \raf{eq:choice-lem-main}}{\leq} \left( \eps^{2k}(\alpha+\eps)^{2k-1} + (\eps/4)^{4k} \right) p (pm)^{2k}\\
\; & \overset{\phantom{\raf{eq:single-O}}}{\leq} (3\eps)^{2k}p(pm)^{2k}.
\end{align*}
With this \raf{eq:wishfor2} follows once we sum over all vertices in $\tilde W$. 

It remains to prove \raf{eq:goal-D}. Suppose $|D_{\alpha+2\eps}(\tilde z)| >\eps^{2k}pm$ 
for some vertex $\tilde z \in \tilde W$, and choose 
$B \subseteq D_{\alpha+2\eps}(\tilde z) \subseteq \tilde V_1$ of size $\lceil \eps^{2k} pm \rceil\leq (\alpha+\eps)pm$. 
Let us now count the number of members of $\cO(\alpha+2\eps,H,\Gamma,\tilde z)$ with the $(\alpha+2\eps)$-saturated edge of the form $\tilde z b$ where $b \in B$. We write $\cO(B,\tilde z)$ to denote this number. By the definition of an $(\alpha+2\eps)$-saturated edge, we attain 
\begin{align*}
|\cO(B,\tilde z)| & \geq |B| p ((\alpha+2\eps)pm)^{2k-1}= \eps^{2k}(\alpha+2\eps)^{2k-1}p(pm)^{2k}.
\end{align*}
On the other hand, \raf{eq:single-O} with $X_{\tilde z} = N_H(\tilde z) \cap \tilde U_1$ and $B$ instead of $D_{\mu}(\tilde z)$, 
together with Claim~\ref{lem:qxqy} with $X=X_{\tilde z}$ and $Y = B$ yield
\begin{align*}
|\cO(B,\tilde z)| & \overset{\phantom{\raf{eq:choice-lem-main}}}{\leq} |B| p ((\alpha+\eps)pm)^{2k-1} + \xi p(pm)^{2k}\; \mbox{(by Claim~\ref{lem:qxqy} and \raf{eq:single-O})}\\
\; & \overset{\raf{eq:choice-lem-main}}{=} \left(\eps^{2k}(\alpha+\eps)^{2k-1} + (\eps/4)^{4k} \right)p (pm)^{2k}.
\end{align*}
The contradiction here is that 
$$
\eps^{2k}(\alpha+\eps)^{2k-1} + (\eps/4)^{4k}  < \eps^{2k}(\alpha+2\eps)^{2k-1}.
$$
To see this, observe that
$$
(\alpha+\eps)^{2k-1}+\frac{\eps^{2k}}{16^{2k}}  \leq (\alpha+\eps)^{2k-1}+(\eps/16)^{2k-1}\leq
\left(\alpha+\eps + \eps/16\right)^{2k-1} < (\alpha+2\eps)^{2k-1}.
$$

This proves \raf{eq:goal-D} and thus completes our proof of Lemma~\ref{lem:main}.\qed

\subsection{Proof of Claim~\ref{lem:qxqy}}\label{sec:qxqy} 
Given $k,\; \xi,\; \alpha,\; \eps,\; \eta$, and $\nu$, we set
\begin{equation}\label{eq:gamma-choice-qxqy}
\gamma = \frac{\xi \left(\log(1+\eta)\right)^{2k}}{2^{8k} \nu}\,,
\end{equation}
choose $n_0$ to be sufficiently large, and let $\Gamma$ be a $(p,\beta)$-jumbled graph, where $\beta$ 
satisfies~\eqref{eq:cl8beta}.

Recall that we seek to show that $\beta Q_{\eta}(X) Q_{\eta}(Y) \leq \xi p(pm)^{2k}$, 
where 
\[Q_{\eta}(X) = \sum_{\vek s \in \cI_{\eta}} \sqrt{|Z_{\eta}(\vek s,X)|}(1+\eta)^{\sum \vek s}\,,\] 
and  $Q_{\eta}(Y)$ defined in a similar manner. To this end, we shall now consider the term 
\begin{equation}\label{eq:express}
q_{\eta}(\vek s,X) = \sqrt{|Z_{\eta}(\vek s,X)|}(1+\eta)^{\sum \vek s},
\end{equation}
where $\vek s \in \cI_{\eta}$. Below we shall prove that for any $\vek s \in \cI_{\eta}$
\begin{equation}\label{eq:small-q-bound}
\mbox{$q_{\eta}(\vek s,X)\leq 2^{4k}p^{k-\frac{1}{2(2k-1)}} m^{\frac{2k-1}{2}}$.}
\end{equation}
This estimate will hold for the counterpart term $q_{\eta}(\vek t,Y)$ with $X$ replaced by $Y$ and $\vek t \in \cI_{\eta}$ as well due to symmetry.

Assuming \raf{eq:small-q-bound}, we prove that $\beta Q_{\eta}(X) Q_{\eta}(Y) \leq \xi p(pm)^{2k}$ as follows. We observe the identity 
\begin{equation}\label{eq:error-small-q}
\beta Q_{\eta}(X) Q_{\eta}(Y) = \beta \sum_{\vek s,\vek t}q_{\eta}(\vek s,X)q_{\eta}(\vek t,Y),
\end{equation}
and put 
$$
L = (2 \log_{1+\eta} n )^{2(k-1)} \geq  L_{\eta}^{2(k-1)},
$$ 
which is an upper bound on the number of summands in \raf{eq:error-small-q} (for $n$ sufficiently large). Then,
\begin{align*}
\beta Q_{\eta}(X)Q_{\eta}(Y) & \overset{\phantom{\raf{eq:gamma-choice-qxqy}}}{\leq} \beta L \left(2^{4k}p^{k-\frac{1}{2(2k-1)}} m^{\frac{2k-1}{2}} \right)^2 \\
& \overset{\phantom{\raf{eq:gamma-choice-qxqy}}}{=} \beta L 2^{8k} p^{2k-\frac{1}{2k-1}}m^{2k-1} \nonumber \\
& \overset{\eqref{eq:cl8beta}}{\leq} \gamma  2^{8k} \left(\frac{1}{\log(1+\eta)}\right)^{2k-1}p^{2k+1} m^{2k-1} n \\
& \overset{\raf{eq:gamma-choice-qxqy}}{\leq} \xi p(pm)^{2k}. 
\end{align*}

It remains to prove \raf{eq:small-q-bound}. Fix now a tuple $\vek s \in \cI_{\eta}$. 
We shall consider two cases. Throughout these cases we shall use the estimates 
\begin{equation}\label{eq:atmost2}
\alpha+\eps \leq 2 \quad \text{and}\quad  1+\eta \leq 2.
\end{equation}
\begin{enumerate}[label=\alabel]
\item Suppose, firstly, that $|Z_{\eta}(\vek s_j,X)| < p^{1/(2k-1)}m$ for all $2\leq j \leq k$.  
 We shall show
\begin{equation}
\label{eq:1eta2}
q_{\eta}(\vek s,X)=\sqrt{|Z_{\eta}(\vek s,X)|}(1+\eta)^{\sum \vek s} \leq (1+\eta)^k\sqrt{|X|}\prod_{j=2}^{k} M_j,
\end{equation}
where
$$
M_j= 2\max\left\{\beta,p \sqrt{|Z_{\eta}(\vek s_j,X)||Z_{\eta}(\vek s_{j-1},X)|}\right\} < 2p^{1+1/(2k-1)}m.
$$

Then, \raf{eq:1eta2} together with the assumption that $\sqrt{|X|} \leq \sqrt{(\alpha+\eps)pm}$ give 
\begin{align}
q_{\eta}(\vek s,X) & \overset{\phantom{\raf{eq:atmost2}}}{\leq} (1+\eta)^k\sqrt{(\alpha+\eps)pm}\left(2p^{1+\frac{1}{2k-1}}m \right)^{k-1} \nonumber \\
& \overset{\raf{eq:atmost2}}{\leq} 2^{2k} \sqrt{pm}\left(p^{1+\frac{1}{2k-1}}m \right)^{k-1} \nonumber \\
& \overset{\phantom{\raf{eq:atmost2}}}{=} 2^{2k} p^{k-\frac{1}{2(2k-1)}} m^{\frac{2k-1}{2}}, \nonumber \end{align}
so that \raf{eq:small-q-bound} holds in this case.

To verify \raf{eq:1eta2}, we first show for all $2\leq j\leq k$ 
\begin{equation}\label{eq:repeat}
\sqrt{|Z_{\eta}(\vek s_{j},X)|}(1+\eta)^{s_{j}-1} \leq M_{j} \sqrt{|Z_{\eta}(\vek s_{j-1},X)|}.
\end{equation} 
Note that \raf{eq:repeat} holds if $(1+\eta)^{s_{j}-1} \leq 2 p |Z_{\eta}(\vek s_{j-1},X)|$. On the other hand, if
$(1+\eta)^{s_j} > 2 p |Z_{\eta}(\vek s_{j-1},X)|$ holds then 
$$
\sqrt{|Z_{\eta}(\vek s_{j},X)|} \overset{\raf{eq:bij_size}}{\leq} \frac{\beta \sqrt{|Z_{\eta}(\vek s_{j-1},X)|}} {(1+\eta)^{s_{j}-1} - p |Z_{\eta}(\vek s_{j-1},X)|} 
\leq \frac{\beta \sqrt{|Z_{\eta}(\vek s_{j-1},X)|}} {\frac{1}{2}(1+\eta)^{s_{j}-1}}.
$$
Repeating \raf{eq:repeat} for each $2\leq j\leq k$ yields \raf{eq:1eta2}. This leads to a $(1+\eta)^{k-1}$ multiplicative factor, here we take $(1+\eta)^{k}$.
This concludes the proof of \raf{eq:1eta2}.

\item Suppose, secondly, that $|Z_{\eta}(\vek s_j,X)| \geq p^{1/(2k-1)}m$ for some $2\leq j \leq k$. To prove \raf{eq:small-q-bound} in this case, we express 
$q_{\eta}(\vek s,X)$ as a product of two numbers, that is, we write   
 \begin{align}
q_{\eta}(\vek s,X) &= \sqrt{|Z_{\eta}(\vek s,X)|}(1+\eta)^{\sum \vek s} = R_1 \times R_2, \label{eq:R1R2}\\ 
\intertext{where $R_1$ is given by}
R_1 & =  \sqrt{|Z_{\eta}(\vek s_j,X)|}(1+\eta)^{\sum \vek s_j}, \label{eq:R1}\\
\intertext{and where $R_2$ is given by} R_2 & = \prod_{r=j}^{k-1} 
\sqrt{\frac{|Z_{\eta}(\vek s_{r+1},X)|}{|Z_{\eta}(\vek s_r,X)|}} (1+\eta)^{s_{r+1}}. \label{eq:R2}
\end{align}
Before proceeding let us, first, observe that $R_2$ is well-defined. Indeed, we are concerned with $q_{\eta}(\vek s,X)$ provided $Z_{\eta}(\vek s,X)$ is nonempty as otherwise $q_{\eta}(\vek s,X) = 0$ and does not contribute to the sum \raf{eq:error-small-q}. Now, by the definition of the $Z_{\eta}$-sets, the set $Z_{\eta}(\vek s,X)$ being nonempty implies that every set $Z_{\eta}(\vek s_r,X)$ is nonempty for each $r \in [k]$. Consequently, it is valid to divide by $|Z_{\eta}(\vek s_r,X)|$ for each $r \in [k]$. 
Second, let us also note that $R_1 \times R_2$ is a telescoping product;  the cardinalities of all $Z_{\eta}$-sets cancel each other with only $|Z_{\eta}(\vek s,X)|$ remaining after all cancelations.  

Now, to upper bound $q_{\eta}(\vek s,X)$ as required in this case, we prove that 
\begin{align} 
R_1 & \leq 4^{j}p^{j-\frac{1}{2(2k-1)}}m^{\frac{2j-1}{2}}, \label{eq:R1-bound}\\
\intertext{and that}
R_2 & \leq (6pm)^{k-j} \label{eq:R2-bound}. 
\end{align}
Owing to \raf{eq:R1R2}, these two estimates imply \raf{eq:small-q-bound} in this case.  
In what follows, we prove the estimates \raf{eq:R1-bound} and \raf{eq:R2-bound}. 

To see \raf{eq:R1-bound}, observe first that 
\begin{equation}\label{eq:R1.1}
|Z_{\eta}(\vek s_j,X)| (1+\eta)^{\sum \vek s_j} \leq |X|\left( (1+\eta)(\alpha+\eps)pm \right)^{j-1} \leq (4pm)^j,
\end{equation}
where the first inequality is due to the degree-regularity of $H$ (see, \eqref{eq:boundpx} and \raf{eq:path-limits}), and the second inequality is due to the assumption that $|X| \leq (\alpha+\eps)pm$.
This together with the assumption of this case that $|Z_{\eta}(\vek s_j,X)| \geq p^{1/(2k-1)}m$ yield that
\begin{equation}\label{eq:R1.2}
(1+\eta)^{\sum \vek s_j} \leq \frac{(4pm)^j}{|Z_{\eta}(\vek s_j,X)|} \leq 4^j p^{j-\frac{1}{2k-1}}m^{j-1}.
\end{equation}
Rewriting \raf{eq:R1}, we arrive at 
\begin{align*}
R_1 &= \sqrt{|Z_{\eta}(\vek s_j,X)| (1+\eta)^{\sum \vek s_j}}(1+\eta)^{\frac{1}{2}\sum \vek s_j}.
\intertext{Owing to \raf{eq:R1.1} and \raf{eq:R1.2}, we then have} 
R_1 & \leq (4pm)^{j/2} \left( 4^j p^{j-\frac{1}{2k-1}}m^{j-1}\right)^{1/2},
\end{align*}
and \raf{eq:R1-bound} follows. 

It remains to prove \raf{eq:R2-bound}. To see this, let us first rewrite \raf{eq:R2} as to attain the form
\begin{equation}\label{eq:R2.1}
R_2 = \prod_{r=j}^{k-1} \left( \frac{|Z_{\eta}(\vek s_{r+1},X)|}{|Z_{\eta}(\vek s_r,X)|}(1+\eta)^{s_{r+1}}\right)^{1/2}(1+\eta)^{\frac{s_{r+1}}{2}}. 
\end{equation}
Recall, first, that for any $r \in [k]$, it holds that $(1+\eta)^{s_r} \leq 8pm$, by \raf{eq:def-deg}. 
Second, for $r \in [k-1]$, observe that 
$$
|Z_{\eta}(\vek s_{r+1},X)|(1+\eta)^{s_{r+1}} \leq (1+\eta)e_H(Z_{\eta}(\vek s_r,X),Z_{\eta}(\vek s_{r+1},X)),
$$  
so that the term ${|Z_{\eta}(\vek s_{r+1},X)|}(1+\eta)^{s_{r+1}}/{|Z_{\eta}(\vek s_r,X)|}$ exceeds the average degree of a vertex in $Z_{\eta}(\vek s_r,X)$ in the graph $H[Z_{\eta}(\vek s_r,X), Z_{\eta}(\vek s_{r+1},X)]$ by a factor of at most $1+\eta$. Owing to the degree-regularity of $H$, this average degree is bounded by $2pm$. Consequently, 
$$
\frac{|Z_{\eta}(\vek s_{r+1},X)|}{|Z_{\eta}(\vek s_r,X)|}(1+\eta)^{s_{r+1}} \leq (1+\eta)2pm \leq 4 pm. 
$$   
It follows that a single factor in \raf{eq:R2.1} is at most 
$\sqrt{4pm} \sqrt{8pm} \leq 6pm$ and \raf{eq:R2-bound} follows. 
\end{enumerate}

This concludes our proof of Claim~\ref{lem:qxqy}. \qed

\end{document}